\documentclass[12pt,oneside]{amsart}
\usepackage{xspace}
\usepackage{graphicx}
\usepackage{latexsym}
\usepackage{mathtools}
\usepackage{amsfonts,amsmath,amscd,amssymb}
\usepackage[all]{xy}

\newcommand{\bA}{{\mathbb A}}

\newcommand{\bF}{{\mathbb F}}

\newcommand{\bK}{{\mathbb K}}

\newcommand{\bZ}{{\mathbb Z}}

\newcommand{\sA}{{\mathcal A}}
\newcommand{\sD}{{\mathcal D}}
\newcommand{\sF}{{\mathcal F}}
\newcommand{\sG}{{\mathcal G}}
\newcommand{\sH}{{\mathcal H}}
\newcommand{\sI}{{\mathcal I}}

\newcommand{\sL}{{\mathcal L}}
\newcommand{\sM}{{\mathcal M}}
\newcommand{\sN}{{\mathcal N}}
\newcommand{\sO}{{\mathcal O}}

\newcommand{\sT}{{\mathcal T}}

\newcommand{\Mo}{{{\mbox{\rm --\,Mod}}}}

\newcommand{\Hom}{{{\mbox{\rm Hom}}}}

\newcommand{\Qc}{{{\mbox{\rm --\,Qcoh}}}}

\newtheorem{theorem}{Theorem}[section]

\newtheorem{lemma}[theorem]{Lemma}
\newtheorem{que}[theorem]{Question}
\newtheorem{cor}[theorem]{Corollary}

\begin{document}
\title{D-affinity and Rational Varieties}
\author{Dmitriy Rumynin}
\address{Department of Mathematics, University of Warwick, Coventry, CV4 7AL, UK
\newline
\hspace*{0.31cm}  Associated member of Laboratory of Algebraic Geometry, National
Research University Higher School of Economics, Russia}
\email{D.Rumynin@warwick.ac.uk}
\thanks{
The work was done entirely during my stay at the Max Planck Institute for Mathematics in Bonn, whose
hospitality and support are gratefully acknowledged.
The research was 
partially supported  by the Russian Academic Excellence Project `5--100'.
The author would like to thank Michel Brion, Masaharu Kaneda, Alexander Kuznetsov and Viatcheslav Kharlamov
for valuable discussions. The author is particularly indebted to Dmitry Kaledin and Adrian Langer for finding gaps in the earlier versions of the paper.}
\date{June 25, 2019}

\begin{abstract}
  We investigate geometry of D-affine varieties.
  Our main result is that a D-affine rational projective surface 
  over an algebraically closed field
  is a generalised flag variety of a reductive group.
\end{abstract}

\maketitle

Let us consider a connected smooth projective algebraic variety $X$
over an algebraically closed field $\bK$ of characteristic zero.
By $G/P$ we denote the generalised flag variety of a reductive algebraic group $G$.

{\bf Beilinson-Bernstein Localisation Theorem} (\cite{BB} for $G/B$, \cite[Th. 3.7]{HoPo} for $G/P$){\bf :} 
{\em If $X \cong G/P$, then $X$ is D-affine.}

It is a long-standing problem whether 
the converse statement holds or there are other smooth projective D-affine varieties
(weighted projective spaces are D-affine but singular \cite{MVB}).
The converse statement is known for toric varieties   \cite{Tho}
and homogeneous varieties \cite{ER}.
  Our main result is the converse statement for the rational surfaces:

  {\bf Main Theorem:}\footnote{Since this paper appeared as a preprint, Langer has shown that in characteristic $0$ or $p>7$ any D-affine surface is $G/P$ \cite{Lan1}. It is a serious improvement of our result.} (Corollary~\ref{cor_surf}.) 
  {\em If $X$ is a D-affine rational surface, then $X \cong G/P$.}

In fact, we aim to cover the most general D-affine varieties 
with various intermediate statements. In particular, many results work in the positive characteristic
as well. The reader should be aware that it is not known which of the partial flag varieties are D-affine
in positive characteristic. Some of them are known to be D-affine:
projective spaces \cite{Haa}, $G/B$ in types $A_2$ \cite{Haa} and $B_2$ \cite{AnKa, Sam1}, odd-dimensional quadrics \cite{Lan}.
On the other hand, the grassmannian $Gr(2,5)$  
is not D-affine \cite{KaLa}. 

There are further notions of D-affinity and derived D-affinity in positive characteristic when instead of Grothendieck differential operators,
either small differential operators \cite{HaKa,KaYe,Kane,Sam2} or crystalline differential operators \cite{BMR,BMR2,Lan1} are studied.
These are not covered by the present paper, although some of our methods may prove useful for these unusual differential operators.

Let us explain the context of the paper section-by-section.
In Section~\ref{s1} 
we define D-affine varieties and make general observations about them and quasicoherent sheaves on varieties.

In Section~\ref{s2} we study divisors on D-affine varieties. 
The main results are Theorems~\ref{1st_th} and \ref{2nd_th}. Both are positivity statements about divisors on a D-affine variety.
We use them to study D-affine surfaces in this section, proving Corollary~\ref{cor_surf}, the main result of the paper.

In Section~\ref{sF_new} we formulate two questions that could pave a way for further research for solving the converse Beilinson-Bernstein problem in general.

\section{Preliminaries}
\label{s1}
Pushing for greater generality of some of our results,
we will work over two algebraically closed field:
a field $\bK$ of characteristic zero
and a field $\bF$ of arbitrary characteristic.

Let $X$ be an algebraic variety over $\bF$,
$\sO_X\Qc$ its  category of quasicoherent sheaves.
Following Kashiwara, we consider 
a sheaf $\sA_X$ of $\bF$-algebras on $X$ together with a morphism
of sheaves of algebras $j: \sO_X \rightarrow \sA_X$
such that $\sA_X$ is a quasicoherent $\sO_X$-module under left multiplication.
Notice that the image of $j$ is not necessarily central, so the left and the right
module structures on $\sA_X$ disagree. The main example of $\sA_X$ are
the structure sheaf $\sO_X$ itself and the sheaf of Grothendieck differential operators $\sD_X$.

We consider the category of $\sO$-quasicoherent left $\sA_X$-modules
$\sA_X\Qc$ (i.e., sheaves of left $\sA_X$-modules, quasicoherent as $\sO_X$-modules).
Notice that the cohomology of an $\sO$-quasicoherent left $\sA_X$-module $\sF$
are independent of the category are independent of the category. Indeed,
$$
\Hom_{\sO_X}(\sO_X , \sF)
\cong
\Gamma (X, \sF)\cong 
\Hom_{\sA_X}(\sA_X , \sF). 
$$
Furthermore, an $\sO$-quasicoherent left $\sA_X$-module admits a resolution
$\sF\rightarrow \sI_0 \rightarrow \sI_1\ldots$
by injective $\sA_X$-modules that are flabby as sheaves \cite[Prop 1.4.14]{HTT} (notice that the reference states this for $\sD_X$ but the proof works for $\sA_X$). Hence,
$$
H^k_{\sA_X}(X , \sF)
\cong
H^k (\Gamma (X, \sI_\bullet))
\cong
H^k(X , \sF)
\cong
H^k_{\sO_X}(X , \sF)
$$
for all $k$. In particular, {\em the acyclicity} of $\sF$
for the functor $\Gamma (X,-)$
(i.e.,
vanishing of $H^k(X , \sF)$ for $k>0$) is independent of the category.
On the other hand, the generation by global sections depends on the category.
The space of global sections $\bA=\Gamma (X,\sA_X)$ is an algebra.
We say that $\sF$
is {\em $\sA$-generated by global sections}  if the following natural map
is surjective:
$$
\sA_X\boxtimes_{\bA} \Gamma (X,\sF) \rightarrow \sF \, .
$$
For instance, $\sA_X$ is always $\sA$-generated by global sections but may or may not be $\sO$-generated by global sections. The following lemma is immediate:
\begin{lemma}
  \label{lemma1}
  \begin{enumerate}
  \item If $\sF$ is $\sO$-generated by global sections, then it is $\sA$-generated by global sections.
    \item The following two statements are equivalent:
    \begin{enumerate}
    \item $\sA_X$ is $\sO$-generated by global sections, 
    \item Any $\sO$-quasicoherent left $\sA_X$-module $\sF$, $\sA$-generated by global sections, is $\sO$-generated by global sections
      \end{enumerate}
          \end{enumerate}
  \end{lemma}

The variety $X$ is called {\em A-affine} if
\begin{itemize}
\item $\Gamma: \sA_X\Qc \rightarrow \bA\Mo$ is exact,
\item if $\sF \in\sA_X\Qc$ and $\Gamma (\sF)\cong 0$, then $\sF\cong 0$.
\end{itemize}

\begin{lemma} \cite{HTT, Ka}
  \label{d-affine}
  The following statements 
  (where statement (5) requires $X$ to be projective with an ample line bundle
  $\sL$)
about an algebraic variety $X$ over $\bF$ are equivalent:
  \begin{enumerate}
  \item $X$ is A-affine.  
  \item    $\Gamma: \sA_X\Qc \rightarrow \bA\Mo$ is an equivalence.
      \item    $\sA_X \boxtimes_\bA \, - \,  : \bA\Mo \rightarrow \sA_X\Qc$ is an equivalence.
  \item    Each sheaf $\sF\in\sA_X\Qc$ is acyclic for the functor $\Gamma (X,-)$ and $\sA$-generated by global sections.
  \item  There exists $N>0$ such that the following two statements hold
    for all $n>N$:
    \begin{enumerate}
\item     $\sA_X (-n)=\sA_X\otimes_{\sO_X} \sL^{-n\otimes}$ is $\sA$-generated by global sections, 
\item  the 
  map $\Gamma (\sA_X (-n)) \otimes_{\bK} \Gamma (\sL^{n\otimes})\rightarrow \Gamma (\sA_X)$ is surjective.
    \end{enumerate}
  \end{enumerate}
\end{lemma}

In the light of the condition (3) of Lemma~\ref{d-affine} it is useful (as also introduced by Langer \cite{Lan}) to call a variety $X$ {\em A-quasiaffine} if each $\sO$-quasicoherent left $\sA_X$-module is $\sA$-generated by global sections. Let us call a variety $X$ {\em A-O-quasiaffine} if each $\sO$-quasicoherent left $\sA_X$-module is $\sO$-generated by global sections.

While an A-affine variety is A-quasiaffine, a punctured affine space $X=\bK^n\setminus\{0\}$, $n\geq 2$ gives an example of not-D-affine D-quasiaffine variety: the $\sD_X$-module $\sO_X$ has a non-vanishing higher cohomology. It would be nice to have examples illustrating D-O-quasiaffinity (cf. Lemma~\ref{lemma1}).

\begin{que}
Find an example of a variety $X$ with a sheaf of algebras $\sA_X$ such that $X$ is A-affine but not A-O-quasiaffine variety.
\end{que}
\begin{que}
  Is it true that a D-affine variety is necessarily D-O-quasiaffine?
\end{que}

The following lemma is straightforward, so we skip a proof:
\begin{lemma}
  \label{gen_gl_sec}
  Let $\sF,\sF^\prime \in\sA_X\Qc$ be $\sA$-generated by global sections.
\begin{enumerate}
\item If   $\sG\in\sA_X\Qc$ is a quotient of $\sF$, 
  then $\sG$ is $\sA$-generated by global sections.
\item If   $Y\subseteq X$ is a closed subscheme and
  $\sF$ is $\sO$-generated by global sections, then
  $\sF|_Y$ is $\sO_Y$-generated by global sections.
\item   If
  $0 \rightarrow \sF^\prime 
  \rightarrow \sG \rightarrow
  \sF \rightarrow   0$
  is an exact sequence in $\sA_X\Qc$, then $\sG$ is $\sA$-generated by global sections.
 \end{enumerate}
\end{lemma}

The third lemma is easy but contains not so well-known terminology, hence, we give a proof.
\begin{lemma}
  \label{normal_Barth}
Let $\sF \in\sA_X\Qc$ be normal in the sense of Barth as an $\sO_X$-module.
If for each $p\in X$ there exists an open neighbourhood $p\in U\subseteq X$
such that $X\setminus U$ is of codimension
at least two and $\sF|_U$
is  $\sA|_U$-generated by global sections,
then $\sF$ is $\sA$-generated by global sections.
\end{lemma}
\begin{proof}
  Let $F = \sF_X (X)$.
  Consider an open $U\subseteq X$ with $X\setminus U$ is of codimension at least two. Normality means that the restriction map $F=\sF (X) \rightarrow \sF (U)$ is an isomorphism \cite[p. 126]{Har}.

  Let $Y$ be the support of the cokernel of the natural map
  $\gamma: \sA_X \boxtimes F \rightarrow \sF$. 
Generation of $\sF|_U$ by global sections means that $U\cap Y=\emptyset$.
Our condition means that no point $p$ belongs to $Y$. Hence, $Y=\emptyset$
and $\gamma$ is surjective. 
\end{proof}

The following well-known observation is sufficient for our ends.
We believe that it is true for singular varieties as well: it should follow from the description of $\sD_X$
as the dual $\sO_X|\bF$-algebra
of the algebroid of functions on the formal neighbourhood of the diagonal  $X\rightarrow X\times X$ \cite[2.4]{Rum}. 
\begin{lemma}
  \label{product}
  \begin{enumerate}
    \item
    If $X$ and $Y$ are schemes over $\bF$, $(\sA_X, j_X: \sO_X\rightarrow \sA_X)$ and $(\sA_Y, j_Y: \sO_Y\rightarrow \sA_Y)$
    are sheaves of algebras with  morphisms of sheaves of algebras such that $\sA_X$ and $\sA_Y$
are quasicoherent    $\sO_X$-module and $\sO_Y$-module
under corresponding left multiplications, then $\sA_X\boxtimes\sA_Y$ is a sheaf of $\bF$-algebras on $X\times Y$
and
$$
j_{X\times Y}: \sO_{X\times Y} \xrightarrow{\cong} \sO_X\boxtimes \sO_Y \xrightarrow{j_X\boxtimes j_Y} \sA_X \boxtimes \sA_Y
$$
is a morphism of sheaves of algebras such that $\sA_X\boxtimes\sA_Y$
is a  quasicoherent    $\sO_{X\times Y}$-module.
\item
If $X$ and $Y$ are smooth varieties over $\bF$,
then the natural map $\varphi_{X,Y}:\sD_X \boxtimes \sD_Y \rightarrow \sD_{X\times Y}$ 
is an isomorphism of sheaves of $\bF$-algebras on $X\times Y$.
\end{enumerate}
\end{lemma}
\begin{proof} Notice that 
  $$
  \sA_X \boxtimes \sA_Y = p_X^*(\sA_X) \otimes_{\sO_{X\times Y}} p_Y^*(\sA_Y)
  $$
  where $p_X: X\times Y\rightarrow X$, $p_Y: X\times Y\rightarrow Y$
  are the projections. It suffices to compute the sections of
  $\sA_X \boxtimes \sA_Y$ on an open affine subset $U\times V \subseteq X\times Y$.
  Let $R\coloneqq \sO_X(U)$, $A\coloneqq \sA_X(U)$, $S\coloneqq \sO_Y(V)$, $B\coloneqq \sA_Y(V)$. Then
  $$
  p_X^*(\sA_X) (U\times V) = A \otimes_{\bF} S, \ \
  p_Y^*(\sA_Y) (U\times V) = R \otimes_{\bF} B 
  $$
  so that we have a natural isomorphism $f$ of $R\otimes_{\bF}S$-modules
  $$
  \sA_X\boxtimes \sA_Y (U\times V) =
  (A \otimes_{\bF} S) \otimes_{R\otimes_{\bF}S} (R \otimes_{\bF} B)
  \xrightarrow{f} A \otimes_{\bF} B
  $$
  given, together with its inverse, by
  $$
  f\big( (a\otimes s)\otimes (r\otimes b) \big) = j_X(r)a \otimes j_Y(s)b, \
  f^{-1}(a\otimes b)  = (a\otimes 1)\otimes (1\otimes b). 
  $$
  Thus, $\sA_X\boxtimes \sA_Y$ is a sheaf of algebras. The map $j_{X\times Y}$
  is given locally by
  $$
  j_{X\times Y} (r\otimes s)  = j_X(r) \otimes j_Y (s) (\mbox{also equal to } (1\otimes s)\otimes (r\otimes 1) \, ). 
  $$

  The second statement also becomes clear on  an open affine subset $U\times V \subseteq X\times Y$. In this case $\sD_X (U)$ and $\sD_Y(V)$ are subalgebras of $\sD_{X\times Y}(U\times V)$ so that the natural map  
  $$
  \varphi_{X,Y}: \sD_X \boxtimes \sD_Y(U\times V) \cong \sD_X (U)\otimes_{\bF}\sD_Y(V)\rightarrow \sD_{X\times Y}(U\times V)
  $$
  is given by multiplication  $\varphi_{X,Y} (a\otimes b) =ab$.
  It is a homomorphism of algebras because $\sD_X (U)$ and $\sD_Y(V)$ commute inside $\sD_{X\times Y}(U\times V)$.

  If $X$ and $Y$ are smooth, the tangent sheaves $\sT_X$ and $\sT_Y$ are locally free
  with local frames $\partial/\partial x_n$ and $\partial/\partial y_m$.
  Then $\sD_X$ and $\sD_Y$ have local bases 
  $\partial^{|\alpha|}/\partial x^\alpha$ and $\partial^{|\beta|}/\partial y^\beta$,
  given by multi-indices $\alpha$ and $\beta$. An easy calculation in this basis shows that $\varphi_{X,Y}$ is an isomorphism. 
  \end{proof}

\begin{lemma}
  \label{cor_product}
  If $X$ and $Y$ are D-affine D-O-quasiaffine varieties over $\bF$
  such the map $\varphi_{X,Y}$ from Lemma~\ref{product} is an isomorphism,
then any open subset $j:U\hookrightarrow X\times Y$ is a D-O-quasiaffine variety.
\end{lemma}
\begin{proof}
Decompose the global sections as a composition of functors
$$
\Gamma:
\sD_{X\times Y}\Qc
\xrightarrow{(X\times Y\rightarrow X)_\ast}
(D(X)\boxtimes \sD_{Y})\Qc
\xrightarrow{\Gamma}
D({X\times Y})\Mo .
$$
The assumed tensor product decomposition 
together with D-affinity of $X$ and $Y$ imply that both functors are
equivalences. Hence, 
$X\times Y$ is a D-affine variety.
It is also D-O-quasiaffine by a combination of Lemma~\ref{lemma1} and Lemma~\ref{product}

If $\sF\in\sD_U\Qc$, its direct images $j_\ast \sF$ in the categories of D-modules, quasicoherent shaves and topological sheaves coincide. By D-affinity of $X\times Y$, $j_\ast \sF$ is $\sO$-generated by global sections. Hence, $\sF$ is $\sO$-generated by global sections.  
\end{proof}

\section{Divisors}
\label{s2}
We go straight to the main result of this section.
\begin{theorem}
\label{1st_th}
Let $X$ be an irreducible D-O-quasiaffine algebraic variety over $\bF$, $Y\subset X$ an effective Cartier divisor. Then there exists $n$ such that for all $m$, divisible by $n$ the tensor power of the normal sheaf $\sN_{Y}^{\otimes m}$ is $\sO_Y$-generated by global sections.
\end{theorem}
\begin{proof} Let $U\coloneqq X\setminus Y$ be the open complement and
  $j: U \hookrightarrow X$ its embedding.
Observe that $j_\ast (\sO_U)=\sO_X(\ast Y)$
  is a $\sD_X$-submodule of the sheaf $\sM_X$ of rational functions and a union of invertible sheaves  $\sO_X(n Y)$.
  On an open subset $V\subseteq X$ (from some cover of $X$) the divisor $Y$ is defined by
  a single function $h$ so that 
  $$
  \sO_X(n Y) (V) =\{ \frac{f}{h^n}\} \, , \ \  
  \sO_X(\ast Y) (V) =\{ \frac{f}{h^k} \}\, , \ \  f\in \sO_X(V), k\in \bZ \, .
  $$
  Since $\sO_X(\ast Y)$ is a $D$-module, it is $\sO$-generated by global sections $s_1, s_2 \ldots$
  Each $s_k$ is a global section of a line bundle $\sO_X(n_k Y)$. Out of the cover $X = \cup_k \{ s_k \neq 0\}$ choose a finite subcover. Let $n$ be the least common multiple of all $n_k$ from the finite subcover. At every point of $X$ one of the sections from the finite subcover does not vanish. It follows that for all $m$ with $n| m$ the invertible sheaf $\sO_X(mY)$ is $\sO$-generated by global sections because one of $s_k^{m/n_k}$ does not vanish at every point. 

  Observe that we have the standard sequence of $\sO_X$-modules
     $$
  0 \rightarrow
  \sO_X \rightarrow
  \sO_X(Y) \xrightarrow{f}
  \sN_{Y\mid X} 
  \rightarrow   0
  $$
  without any further restrictions on $X$ or $Y$,
  where $\sN_{Y\mid X}$ is the normal sheaf of the embedding $Y\hookrightarrow X$,
  considered as a sheaf on $X$. 
  This yields another exact sequence
    $$
  0 \rightarrow
  \ker f^{\otimes m} \rightarrow
  \sO_X(mY) \xrightarrow{f^{\otimes m}}
  \sN_{Y\mid X}^{\otimes m} 
  \rightarrow   0
  $$
  for all $m$.
  By Lemma~\ref{gen_gl_sec}, if $n| m$, then 
  $\sN_{Y\mid X}^{\otimes m}$ is $\sO$-generated by global sections. Hence, so is its restriction
  $\sN_{Y}^{\otimes m} = (\sN_{Y\mid X}^{\otimes m}) |_Y$.
\end{proof}

We can derive some geometric consequences of $D$-affinity as soon
as we can exhibit some interesting $\sD_X$-modules. For example,
$\sO_X$ is a $\sD_X$-module, thus, if $X$ is complete, we know
some of its Hodge numbers
$$
h^{0,0}(X)=1, \
h^{0,m}(X)=0 \ \mbox{ for } \ m>0.  
$$
For a smooth projective surface $X$ 
this means that $p_a=p_g=0$. If $X$ is also D-O-affine, 
Theorem~\ref{1st_th} implies that the surface is minimal
in a strong sense: $Y^2\geq 0$ for any curve $Y\subseteq X$. 
Moreover, 
$$
c_2(X) = 2 + h^{2,2}(X), \
c_1^2 (X) = 10 - h^{2,2}(X), \
1 \leq h^{2,2}(X)\leq 10.
$$
In the light of the next theorem, it would be interesting to classify minimal models with such numerical invariants that do not have any negative curves.
\begin{theorem}
\label{2nd_th}
  A projective D-affine variety $X$ over $\bF$ of dimension at least 2 cannot have any contractible divisors.
\end{theorem}
\begin{proof}
  Let $C\subseteq X$ be a contractible divisor. Let $\widetilde{X}$ be the blow-down of $X$ at $C$. The centre of the blow-up $X\rightarrow\widetilde{X}$ is a subscheme $Y\subseteq\widetilde{X}$ supported at a finite number of points. We can pick a divisor $D\subseteq \widetilde{X}$ such that $Y\not\subset D$ and $\widetilde{X}\setminus D$ is affine. Let $U= f^{-1} (\widetilde{X}\setminus (D\cup Y))$. Then $U$ is quasiaffine but not affine, while the complement $X\setminus U$ is a divisor. This contradicts Thomsen's Theorem that states if the complement of a divisor on a D-affine variety is quasiaffine, then it is affine  \cite[Lemma 1]{Tho}.
Notice that while Thomsen also assumes smoothness, it is never used in the proof.
\end{proof}

We are ready to prove the main theorem:
\begin{cor}\label{cor_surf}
  A rational smooth connected projective D-affine surface $X$ over $\bF$
is isomorphic to either $P^2$ or $P^1 \times P^1$. 
\end{cor}
\begin{proof}
  By Theorem~\ref{2nd_th}, $X$ is minimal.
A minimal smooth rational surface is either $P^2$ or the Hirzebruch surface $H_n$, $n\geq 0$.
Since $H_n$ contains an irreducible curve $C$ with $C^2=-n$, we conclude that $n=0$.
Finally, $H_0\cong P^1 \times P^1$.
\end{proof}

Given an arbitrary closed subvariety $Z\subseteq X$, we can produce some $\sD_X$-modules
supported on $Z$, for instance, functions on the formal neighbourhood of $Z$
or local cohomology sheaves $\sH^n_Z (\sF)$ where $\sF$ is an $\sD_X$-module,
e.g., $\sF =\sO_X$ or $\sF =\sD_X\otimes_{\sO_X}\sL$ for a line bundle $\sL$. It would be interesting to analyse how affinity
of these sheaves affects geometry of $X$.

It is also interesting to resolve D-affinity of some particular varieties:
\begin{que}
  Can a fake projective space be D-affine?
\end{que}
\begin{que}
  Can a cubic hypersurface in $P^n$, $n>3$ be D-affine?
\end{que}

\section{Two further questions}
\label{sF_new}
We would like to state two further questions that could be quite useful for further research, including future attempts to settle the inverse Beilinson-Bernstein problem.

\begin{que}
  \label{queN}
  Let $X$ be an irreducible D-affine algebraic variety over $\bF$, $Y\subset X$ an effective Cartier divisor. 
Is $\sN_{Y}$ necessarily 
$\sO_Y$-generated by global sections?
\end{que}

\begin{que}
  \label{queF}
  Characterize the class of varieties $X$ such that for each point $p\in X$ there exists an open set $U\subseteq X$ such that
  \begin{itemize}
  \item $p\in U$,
  \item the complement $X\setminus U$ has a codimension at least 2,
    \item $\delta :U\rightarrow U^2$ is a scheme-theoretic complete intersection.
  \end{itemize}
  \end{que}

\end{document}